\documentclass[a4paper]{cas-sc}
\usepackage{amsmath,amssymb,amsthm,mathtools}
\usepackage{algorithm}
\usepackage{algpseudocode}
\usepackage{booktabs}
\usepackage{microtype}
\usepackage{pgfplots}
\pgfplotsset{compat=1.17}
\usepackage{hyperref}
\usepackage{cleveref}

\usepackage{pgfplotstable}
\usepackage{caption}
\usepackage{subcaption}
\usepgfplotslibrary{groupplots}

\usepackage{xcolor}
\definecolor{darkyellow}{RGB}{200,160,0}
\definecolor{darkgreen}{RGB}{50,175,50}

\usepackage{pgfplots}
\usepackage{pgfplotstable}

\usepackage[numbers]{natbib}

\def\tsc#1{\csdef{#1}{\textsc{\lowercase{#1}}\xspace}}
\tsc{WGM}
\tsc{QE}

\newtheorem{theorem}{Theorem}[section]
\newtheorem{lemma}{Lemma}[section]
\newtheorem{problem}{Problem}[section]
\newtheorem{proposition}{Proposition}[section]
\newtheorem{corollary}{Corollary}[section]

\newtheorem{example}{Example}[section]
\newtheorem{remark}{Remark}[section]
\newtheorem{conjecture}{Conjecture}[section]

\numberwithin{figure}{section}

\newcommand{\bbR}{\mathbb{R}} 
\newcommand{\bbC}{\mathbb{C}} 
\newcommand{\diag}{\mathrm{diag}}

\begin{document}

\let\WriteBookmarks\relax
\def\floatpagepagefraction{1}
\def\textpagefraction{.001}

\shorttitle{Schur Stability of Image Reconstruction Operators}    

\shortauthors{D.~Banerjee and K.~N.~Chaudhury}  

\title [mode = title]{On the Schur Stability of Some Image Reconstruction Operators}  

\author[1]{Debraj~Banerjee}[orcid={0009-0005-8991-8903}]

\cormark[1]

\ead{debrajb@iisc.ac.in}

\author[1]{Kunal~N.~Chaudhury}[orcid={0000-0002-8136-605X}]

\ead{kunal@iisc.ac.in}

\credit{Supervision, Funding Acquisition, Writing - Review & Editing}

\affiliation[1]{organization={Department of Electrical Engineering, Indian Institute of Science},
            city={Bengaluru},
            postcode={560012}, 
            state={Karnataka},
            country={India}}

\cortext[1]{Corresponding author}

\begin{abstract}
We investigate an open problem arising in iterative image reconstruction. In its general form, the problem is to determine the stability of the parametric family of operators $P(t) = W (I-t B)$ and $R(t) = I-W + (I+tB)^{-1} (2W-I)$, where $W$ is a stochastic matrix and $B$ is a nonzero, nonnegative matrix. We prove that if $W$ is primitive, then there exists  $T > 0$ such that the spectral radii $\varrho(P(t))$ and $\varrho(R(t))$ are strictly less than $1$ for all $0 < t < T$. The proof combines standard perturbation theory for eigenvalues and an observation about the analyticity of the spectral radius. This argument, however, does not provide an estimate of $T$. To this end, we compute $T$ explicitly for specific classes of $W$ and $B$. Building on these partial results and supporting numerical evidence, we conjecture that if $B$ is positive semidefinite and satisfies certain technical conditions, then $\varrho(P(t)), \, \varrho(R(t))<1$ for all $0 < t <  2/\varrho(B)$. As an application, we show how these results can be applied to establish the convergence of certain iterative imaging algorithms.
\end{abstract}


\maketitle

\section{Introduction}
\label{sec:intro}

We study a linear algebra problem motivated by iterative image reconstruction. We begin by stating the problem, highlighting the main difficulties, and explaining why existing results do not apply. We then present our results along with a conjecture. Finally, we explain how these findings connect to the convergence of iterative algorithms.

Throughout, we work in $\bbC^n$ or $\bbR^n$ for some fixed $n \geqslant 1$. We let $I \in \bbR^{n \times n}$ denote the identity matrix,  $e \in \mathbb{R}^n$ the all-ones vector, and $\varrho(M)$ the spectral radius of $M \in \bbC^{n \times n}$,
\begin{equation*}
\varrho(M) = \max \big\{ |\lambda| : \ \lambda \in \sigma(M) \big\},
\end{equation*}
where $\sigma(M)$ denotes the spectrum (set of eigenvalues) of $M$. We say $M$ is Schur stable (or simply \emph{stable}) if $\varrho(M) < 1$~\cite{Fleming2000SchurDStable}. For $v \in \bbR^n$, we write $v \geqslant 0$ if all components of $v$ are nonnegative, and $v > 0$ if all components are strictly positive. For $M \in \bbR^{n \times n}$, we write $M \geqslant 0$ if all its entries are nonnegative. 

We call $W \in \mathbb{R}^{n \times n}$ row-stochastic (or simply stochastic) if $W \geqslant 0$ and $We = e$, and doubly stochastic if, in addition, $e^\top W = e^\top$. We say $W$ is irreducible if, for every $i,j$, there exists $m \geqslant 1$ such that $(W^m)_{ij} > 0$, and primitive if a single such $m$ works for all $i, j$. If $W$ is an irreducible stochastic matrix, then it follows from the Perron-Frobenius theorem that $\varrho(W) = 1$ and $1$ is a simple eigenvalue. The vector $e$ is a right eigenvector, and there exists a positive left (Perron) eigenvector $\pi$ such that $\pi^\top e = 1$. Moreover, if $W$ is primitive, then $|\lambda| < 1$ for $\lambda \in \sigma(W), \lambda \neq 1$~(e.g., see, \cite{meyer2000matrix}).

Let $W \in \bbR^{n \times n}$ be a stochastic matrix, and let $B \in \bbR^{n \times n}$ be a nonzero matrix. For $t \geqslant 0$, consider the family of matrices
\begin{equation}
\label{def:P}
P(t) = W(I - tB),
\end{equation}
and
\begin{equation}
\label{def:R}
R(t) = I-W + (I+tB)^{-1} (2W-I),
\end{equation}
where we assume that the inverse $(I + tB)^{-1}$ exists. The problem is to estimate $\varrho(P(t))$ and $\varrho(R(t))$ for different values of $t$, subject to additional assumptions on $W$ and $B$. 

\begin{problem}
\label{prob:existence}
Under what assumptions on $W$ and $B$, do there exist $t_1, t_2 > 0$ such that $P(t_1)$ and $R(t_2)$ are stable? More generally, can we find $T > 0$ such that $P(t)$ and $R(t)$ are stable for all $0 < t < T$?
\end{problem}

Standard techniques for estimating the spectral radius include bounds based on matrix norms, Gershgorin’s theorem, Gelfand’s formula, the Collatz–Wielandt formula, and the Bauer–Fike theorem~\cite{meyer2000matrix}. However, these tools are difficult to apply to matrix products. For instance, $W$ and $I-tB$ can each have spectral radius $1$, yet their product can still be stable. Conversely, even if both $W$ and $I-tB$ are stable, their product need not be stable. In what follows, we collect some observations relevant to~\Cref{prob:existence}.

\begin{remark}
\label{rmk:relatedwork}
The problem of bounding the spectral radius of matrix products has been studied in the literature, but the existing results are limited and do not directly apply to~\Cref{prob:existence}. For instance,~\cite{Axtell2009,kozyakin2017minimax} addresses only nonnegative matrices. While $W \geqslant 0$,  $I - tB$ may contain negative components.  There are also Varga-type theorems on the stability of matrix products~\cite{Varga2000MatrixIterative,noutsos2006perron}. In particular, we can rewrite \eqref{def:P} as
\begin{equation*}
P(t) = - F V^{-1}, \quad F:=W, \, V = - (I-tB)^{-1},
\end{equation*}
where $(I-tB)^{-1}$ exists and is nonnegative for small $t$. Varga's theorem guarantees the stability of $-FV^{-1}$ provided that both $F+V$ and $V$ are Metzler~(i.e., have nonnegative off-diagonal components), and if $V$ is Hurwitz. However, in our case, neither $V$ nor $F+V$ is Metzler, so Varga’s criterion does not apply. Similarly, \eqref{def:R} can be expressed as
\begin{equation*}
R(t)^\top  = -FV^{-1}\,,\qquad F:=  (I-W^\top)(I+tB^\top) + (2W^\top-I),\ V := -(I+tB^\top).
\end{equation*}
In this case, V also fails to be Metzler.
\end{remark}

\begin{remark}
Since $P(0) = R(0)=W$ and $\varrho(W)=1$, we cannot use the continuity of the spectral radius to conclude that $P(t)$ and $R(t)$ are stable in a neighbourhood of $0$. On the other hand, $P(t)$ is not stable for large $t$. Indeed, if $WB$ is not nilpotent, we have from the continuity of $\varrho$ that
\begin{equation*}
\lim_{t \to \infty} \: \varrho(P(t)) = \lim_{t \to \infty} \: t  \, \varrho \left(\frac{1}{t} W - WB \right) = \infty.
\end{equation*}
As for $\varrho(R(t))$, predicting its behavior at infinity is not possible without further assumptions on $B$.
\end{remark}

\begin{remark}
\label{rmk:necessary}
Since $We=e$, it follows from~\eqref{def:P} that $P(t)\, e = e$ if $Be=0$. 
Consequently, $\varrho(P(t)) \geqslant 1$ for all $t >0$. Thus, it is necessary that $Be \neq 0$ for $P(t)$ to be stable for some $t > 0$. Moreover, even if $Be \neq 0$, $P(t)$ may fail to be stable if $W$ is not primitive. For example, consider the stochastic matrix
\begin{equation}
\label{eq:example_W_irr}
W = 
\begin{pmatrix} 0 & 1 \\ 
1 & 0 \end{pmatrix}.
\end{equation}
This is not primitive. If we choose  $B=W$, then 
\begin{equation*}
P(t) = W (I-tW) = W - t W^2 = W- t I.
\end{equation*}
A simple calculation shows that the eigenvalues of $P(t)$ are $-(1+t)$ and $1-t$; hence, $P(t)$ is not stable for any $t >0$. 

The condition $Be \neq 0$ is also necessary for the stability of $R(t)$. Indeed, if $Be=0$, then $(I + tB)^{-1}e=e$, which implies $R(t)\, e=e$; hence, $R(t)$ cannot be stable. As with $P(t)$, if $W$ is not primitive, stability cannot be guaranteed. For instance, let $W$ be as in~\eqref{eq:example_W_irr} and take
\begin{equation}
\label{eq:example_B=E/2}
B = 
\frac{1}{2}\begin{pmatrix} 1 & 1 \\ 
1 & 1 \end{pmatrix}.
\end{equation}
In this case,
\begin{equation*}
R(t) = 
\frac{1}{2(t+1)}
\begin{pmatrix}
-t & t+2 \\
t+2 &-t
\end{pmatrix},
\end{equation*}
whose eigenvalues are $-1$ and $1/(t+1)$. Hence, $R(t)$ is not stable for any $t>0$.
\end{remark}

Our main result is that, in addition to the stochasticity of $W$, the above necessary conditions on $W$ and $B$ are (almost) sufficient to resolve \Cref{prob:existence}. This comes as a consequence of a more general result.

\begin{theorem}
\label{thm:existence}
Let $W$ be primitive and stochastic, and let $B$ satisfy $\pi^\top \!Be > 0$, where $\pi$ is the left Perron eigenvector of $W$. Then there exists $T>0$ such that 
$P(t)$ and $R(t)$ are stable for all $0 < t <T$.
\end{theorem}

We have the following corollaries of \Cref{thm:existence}.

\begin{corollary}
\label{corr1:existence}
Let $W$ be primitive and stochastic, and let $B$ be such that $Be \geqslant 0$ and $Be \neq 0$. Then there exists $T>0$ such that $P(t)$ and $R(t)$ are stable for all $0 < t <T$.
\end{corollary}

\begin{corollary}
\label{corr2:existence}
Let $W$ be primitive and stochastic, and let $B \geqslant 0$ with $B \neq 0$. Then there exists $T>0$ such that 
$P(t)$ and $R(t)$ are stable for all $0 < t <T$.
\end{corollary}

\Cref{thm:existence} is proved in~\Cref{sec:proofs}  using complex analytic techniques, specifically the perturbation theory of eigenvalues~\cite{kato1995}. That complex analysis comes up is not surprising since $P(t)$ and $R(t)$ are generally nonsymmetric and can have complex eigenvalues.

\begin{remark}
\label{rmk:counterexample1}
We illustrate the importance of the condition  $\pi^\top Be >0$ in~\Cref{thm:existence}. Let
\begin{equation}
\label{eq:counterexample1}
W = \frac{1}{10}\begin{pmatrix}
    7 & 3 \\
    6 & 4
\end{pmatrix}
\quad \mbox{and} \quad
B = \frac{1}{10}\begin{pmatrix}
    34 & -65 \\
    -65 &  126
\end{pmatrix}.
\end{equation}
The matrix $W$ is primitive and stochastic, with $\pi^\top = (2/3 \ \ 1/3)$. In this case, $\pi^\top Be = -1/30 < 0$, and we find that $\varrho(P(t))>1$ and $\varrho(R(t)) > 1$ for all $0 < t < 1/2$.
\end{remark}

\begin{remark}
\label{rmk:B_blur}
Perturbation theory does not provide an estimate for $T$, which is required for selecting the value of $t$ in applications. 
To get an idea about how large $T$ can be, we consider the example:
\begin{equation}
\label{eq:example_W_B_blur}
W = \begin{pmatrix}
   0.3 & 0.7 \\ 0.6 & 0.4
\end{pmatrix}, \quad 
H= \begin{pmatrix}
   0.913 & 0.087 \\ 0.087 & 0.913
\end{pmatrix}, \quad \mbox{and} \quad B = H^\top H.
\end{equation}
Clearly, $W$ and $B$ satisfy the assumptions in~\Cref{thm:existence}. 
The matrix $H$ acts as a blur (averaging) operator, and the particular choice of $B$ comes from a practical application discussed in \Cref{sec:application}. The spectral radii are shown in~\Cref{fig:blur}. We see that $P(t)$ is stable over the interval $(0,2)$, whereas $R(t)$ remains stable even beyond this range.
\end{remark}

\begin{figure}[t]
\centering
\begin{subfigure}[t]{0.47\textwidth}
   \centering
   \includegraphics[width=1.0\linewidth]{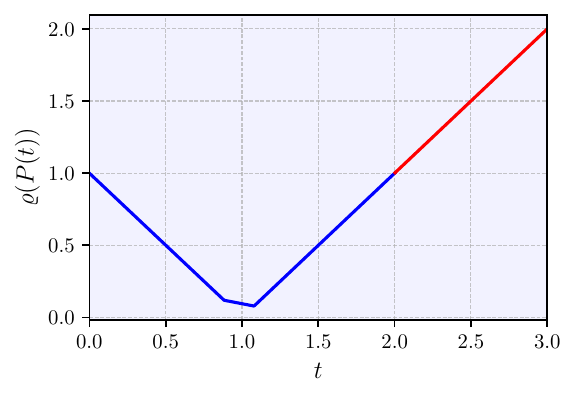}
\end{subfigure}
\hfill
\begin{subfigure}[t]{0.47\textwidth}
   \centering
   \includegraphics[width=1.0\linewidth]{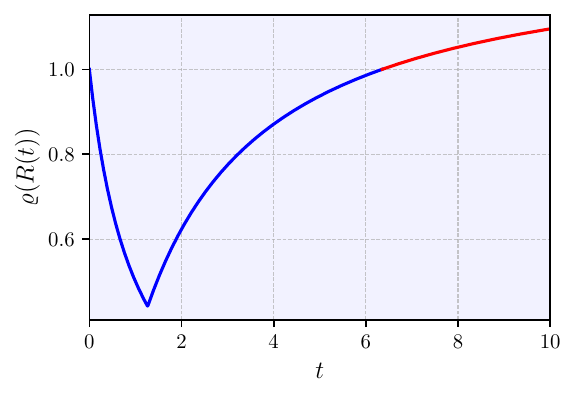}
\end{subfigure}
\caption{Spectral radii of $P(t)$ and $R(t)$ in~\Cref{rmk:B_blur}, illustrating the existence of a threhsold $T$ in~\Cref{thm:existence}. Blue indicates the stable regime ($\varrho < 1$), and red indicates instability ($\varrho \geqslant 1$). In both cases, there is a threshold below which the operators are stable and beyond which they become unstable.}
\label{fig:blur}
\end{figure}

The threshold $T=2$ in~\Cref{rmk:B_blur} matches $2/\varrho(B)$ since $\varrho(B)=1$. In fact, this threshold can be predicted when $W$ is symmetric and $B$ is positive semidefinite.  In particular, we have the following result, adapted from~\cite{ACK2023-contractivity}.

\begin{theorem}[\cite{ACK2023-contractivity}]
\label{thm:contractivity}
Let $W$ be symmetric, primitive, and stochastic, and $B$ be positive semidefinite with $Be \neq 0$. Then $P(t)$ and $R(t)$ are stable for all $0 < t < 2/\varrho (B)$.
\end{theorem}

In fact, a stronger result was established in~\cite{ACK2023-contractivity}—namely, it was shown that the spectral norm of $P(t)$ and $R(t)$ is strictly less than $1$, which implies stability. \Cref{thm:contractivity} relies on the fact that when $W$ and $I-tB$ are symmetric, they can be diagonalized in orthonormal bases. We cannot use this as $W$ can be nonsymmetric in the present setting. Moreover, the proof in~\cite{ACK2023-contractivity} relies on the submultiplicativity of the spectral norm, which generally fails for the spectral radius.

A natural question is whether symmetry is necessary in \Cref{thm:contractivity}. Note that if $W$ is symmetric and stochastic, it is automatically doubly stochastic.  In fact, \Cref{thm:contractivity} can be formulated under this weaker assumption.

\begin{theorem}
\label{thm:doublystochastic}
Let $W$ be doubly stochastic such that both $W^\top W$ and $W W^\top $ are irreducible, and let $B$ be positive semidefinite with $Be \neq 0$. Then $P(t)$ and $R(t)$ are stable for all $0 < t < 2/\varrho (B)$.
\end{theorem}

Note that~\Cref{thm:contractivity} follows from~\Cref{thm:doublystochastic}. Indeed, if $W$ satisfies the assumptions in~\Cref{thm:contractivity}, then $W$ is doubly stochastic, and $W^\top W = WW^\top= W^2$ is primitive (and hence irreducible) since $W$ is primitive. These are exactly the conditions in \Cref{thm:doublystochastic}.

\Cref{thm:existence,thm:doublystochastic} show that the symmetry of $W$ assumed in \Cref{thm:contractivity} is not necessary for stability. Henceforth, we assume that $W$ is stochastic but nonsymmetric. The question is whether a corresponding estimate for $T$ can be derived in this setting, analogous to the symmetric case. We present some special cases where such an estimate can be obtained. Moreover, the bound coincides with that for the symmetric case \Cref{thm:contractivity}. One such case arises in the image reconstruction problem of \emph{inpainting}, where $B$ is diagonal (see \Cref{sec:application}).

\begin{proposition}
\label{prop:inpainting}
Let $W$ be irreducible and stochastic,  and let $B \geqslant 0$ be nonzero and diagonal. Then $P(t)$ and $R(t)$ are stable for all $0 < t < 2/\varrho (B)$.
\end{proposition}

The proof relies on linear algebra techniques, as estimating $T$ using just perturbation theory is difficult. We also note that in \Cref{prop:inpainting}, it is sufficient for $W$ to be irreducible, which is a weaker condition than primitivity.

\begin{remark}
Under the assumptions of~\Cref{prop:inpainting}, we have  $\varrho(W)=1$ and $\varrho(I-tB) \leqslant 1$. Moreover, since $I-tB$ is diagonal, the spectral radius is submultiplicative:~$\varrho(P(t)) \leqslant \varrho(W) \, \varrho(I-tB) \leqslant 1$.
A similar argument applies to $R(t)$. Therefore, the nontrivial part of~\Cref{prop:inpainting} is to show that the bound is strictly less than $1$.
\end{remark}

Next, we consider the case where $B$ is not diagonal. This setting includes the \emph{deblurring} problem as a special case (see~\Cref{sec:application}).

\begin{proposition}
\label{prop:nondiagB}
Let $W$ be primitive and stochastic, and let $B = \alpha I + \beta E$ be positive semidefinite with $Be \neq 0$, where $E \in \bbR^{n \times n}$ is the all-ones matrix. Then $P(t)$ and $R(t)$ are stable for all $0 < t < 2/\varrho (B)$.
\end{proposition}

We see from the above results that nonnegativity of $B$ is not essential; instead, the key property appears to be the positive definiteness of $B$. In particular, drawing on \Cref{thm:existence,thm:contractivity,thm:doublystochastic} and \Cref{prop:inpainting,prop:nondiagB}, together with extensive numerical experiments (not reported here), we propose the following conjecture.

\begin{conjecture}
\label{conjecture}
Let $W$ be primitive and stochastic, and let $B$ be positive semidefinite with $Be \leqslant \varrho(B) e$ and $\pi^\top Be > 0$, where $\pi$ is the left Perron eigenvector of $W$. Then $P(t)$ and $R(t)$ are stable for all $0 < t < 2/\varrho (B)$.
\end{conjecture}

The assumptions in \Cref{conjecture} guarantee that $\varrho(W), \varrho(I - tB) \leqslant 1$. Yet, this by itself does not ensure $\varrho(P(t)) \leqslant 1$. We are conjecturing a stronger property that not only is $\varrho(P(t)) \leqslant 1$, but in fact $P(t)$ is stable.

\begin{example}
\label{ex:counterexample2}
We give an example showing that the condition $Be \leqslant \varrho(B)e$ is necessary in~\Cref{conjecture}, even if all other assumptions hold. Let
\begin{equation}
\label{eq:counterexample2}
W = \frac{1}{2}\begin{pmatrix}
    0 & 2 \\
    1 & 1
\end{pmatrix}
\,, \quad
H = \frac{1}{100}\begin{pmatrix}
    48 & 52 \\
    52 & 48
\end{pmatrix}\,,\quad 
S = \begin{pmatrix}
    1 & 0
\end{pmatrix}\,,\quad 
B = (SH)^\top (SH).
\end{equation}
The matrix $W$ is primitive and stochastic with left Perron vector $\pi^\top = (1/2 \quad  1)$. The choice of $B$  corresponds to a practical application discussed in \Cref{sec:application}. In this case, $B$ is positive semidefinite with $\pi^\top Be  > 0$, but the condition $Be \leqslant \varrho(B)e$ is not satisfied. The bound is $2/\varrho(B)=3.9936...$, yet we can verify that $P(t)$ is not stable for $3.86 < t < 2/\varrho(B)$.
\end{example}

Unlike $P(t)$, predicting a strict stability threshold $T$ for $R(t)$ is hard. We explain this with a numerical example.

\begin{example}
\label{ex:admm}
Let $W$ be as in \eqref{eq:example_W_B_blur}, and consider two different choices of $B$:
\begin{equation}
\label{eq:exampleADMM}
B_1 = \frac{1}{10}
\begin{pmatrix}
   30 & 0 \\ 0 & 5
\end{pmatrix}\quad \mbox{and} \quad 
B_2 = \frac{1}{10}\begin{pmatrix}
   4 & -1 \\ -1 & 2
\end{pmatrix},
\end{equation}
where $B_1$ is diagonal and $B_2$ has positive and negative components. Both satisfy the assumptions in~\Cref{conjecture}. The spectral radius of $R(t)$ is shown in~\Cref{fig:admm}. For $B_1$, stability holds for $0 < t < T_1$ with $T_1 = 4.777...$, whereas for $B_2$, it extends to $0 < t < T_2$ with $T_2 = 11.904...$. In both cases, the stability threshold is strictly larger than the bounds $2/\varrho(B_1) = 0.67...$ and $2/\varrho(B_2)=4.5308...$. Moreover, the values of $T_1$ and $T_2$ differ significantly. This shows that the stability threshold depends on $B$ and that $2/\varrho(B)$ underestimates the maximal threshold $T$.
\end{example}

\begin{figure}[t]
\centering
\begin{subfigure}[t]{0.47\textwidth}
   \centering
   \includegraphics[width=1.0\linewidth]{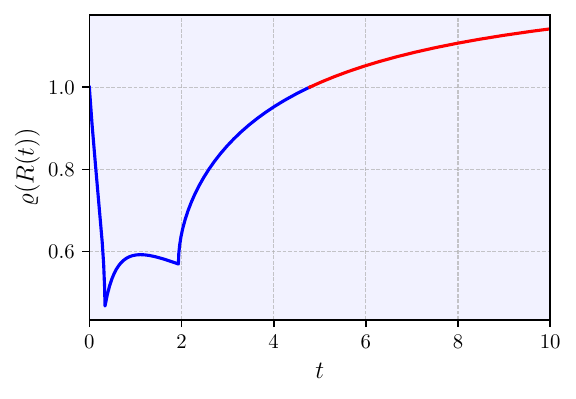}
   \caption{For $B = B_1$ in \eqref{eq:exampleADMM}.}
\end{subfigure}
\hfill
\begin{subfigure}[t]{0.47\textwidth}
   \centering
   \includegraphics[width=1.0\linewidth]{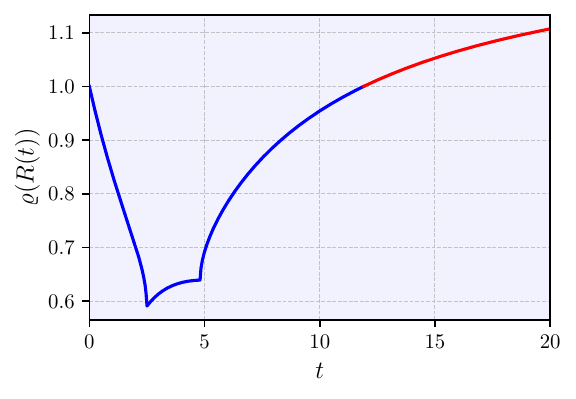}
   \caption{For $B = B_2$ in \eqref{eq:exampleADMM}.}
\end{subfigure}
\caption{Spectral radius of $R(t)$ in~\Cref{ex:admm} for different choices of $B$. Blue indicates the stable regime ($\varrho < 1$), and red indicates instability ($\varrho \geqslant 1$). We observe that the stability threshold is sensitive to the choice of $B$, and that the bound $2/\varrho(B)$ does not coincide with the maximal threshold (see \Cref{ex:admm} for details).}
\label{fig:admm}
\end{figure}

\section{Application}
\label{sec:application}

We now apply the stability results from the previous section. As observed in~\cite{ACK2023-contractivity}, these results provide convergence guarantees for certain iterative algorithms in image reconstruction. For completeness and to place the stability results in a broader context, we briefly outline the main ideas.

The operators~\eqref{def:P} and~\eqref{def:R} come from the plug-and-play method, which uses smoothing operators to regularize ill-posed imaging problems~\cite{sreehari2016plug}. The plug-and-play framework itself is grounded in convex optimization algorithms, such as proximal gradient descent (PGD) and the alternating direction method of multipliers (ADMM)~\cite{Beck2017}. The origins of these algorithms can, in turn, be traced to operator splitting methods~\cite{GabayMercier1976,LionsMercier1979} and, further still, to the numerical solution of partial differential equations~\cite{DouglasRachford1956,PeacemanRachford1955}. 

The operators $P(t)$ and $R(t)$ are derived from PGD and  ADMM in the setting of linear inverse problems~\cite{bouman2022foundations}, namely problems of the form $Ax=b$, where $x \in \bbR^n$ is the reconstructed image, $A \in \bbR^{m \times n}$ is the forward operator (or imaging model), and $b \in \bbR^{m}$ is the observed image. Since $b$ is typically noisy and $A$ is ill-conditioned, the inversion is performed by solving the regression problem
\begin{equation}
\label{eq:loss}
\min_{x \in \bbR^n} \: J(x) = \frac{1}{2} \|Ax-b\|^2,
\end{equation}
where $\| \cdot \|$ is the Euclidean norm on $\bbR^m$ and $J: \bbR^n \to \bbR$ is a quadratic loss function. 

\begin{remark}[properties of $A$]
\label{rmk:propertiesA}
Depending on the application, the forward operator $A$ takes different forms~\cite{bouman2022foundations}. In inpainting, $A$ is a diagonal matrix with components in $\{0,1\}$; in deblurring, $A=H$ is a circulant stochastic matrix; and in superresolution, $A=SH$, where $H$ is as in deblurring and $S \in \bbR^{m \times n}$ is a subsampling matrix formed by selecting $m$ rows from the identity matrix $I \in \bbR^{n \times n}$. In particular, $A$ may be assumed to be nonzero in these applications.
\end{remark}

The regression problem~\eqref{eq:loss} can be solved using gradient descent, with iterations given by
\begin{equation}
\label{eq:GD}
x^{(0)} \in \bbR^n, \quad x^{(k+1)} =  x^{(k)} - t\, \nabla \! J(x^{(k)}), \ k \geqslant 0,
\end{equation}
where $t > 0$ is the step size. Substituting the loss function from~\eqref{eq:loss} yields
\begin{equation*}
 x^{(k+1)}  = (I-tB)\,  x^{(k)} + t  A^\top b \qquad (B:=A^\top \! A).
\end{equation*}

\begin{remark}[properties of $B$]
\label{rmk:propertiesB}
The matrix $B=A^\top \! A$ is always positive semidefinite. For inpainting, deblurring, and superresolution, it follows from~\Cref{rmk:propertiesA} that $A \geqslant 0$, and consequently so is $B$. Furthermore, $Be \neq 0$, since $Be= 0$ would imply $Ae = 0$, which is not possible as the measurement of a constant-intensity image (represented by $e$) cannot be zero. Thus, $B$ satisfies the assumptions in the results in~\Cref{sec:intro}.
\end{remark}

It is well known that the reconstruction obtained through linear regression is typically of poor quality. Thus, some form of \emph{regularization} is required to obtain high-quality images. The plug-and-play framework achieves this by applying a denoiser after each gradient step. Specifically, if a linear denoiser $W$ is used, the iterations~\eqref{eq:GD} take the form
\begin{equation}
\label{eq:PnP}
 x^{(k+1)}  = W \left( x^{(k)} - t \,\nabla \! J(x^{(k)}) \right) = P(t)\,  x^{(k)} + t W A^\top b,
\end{equation}
where $P(t)$ is as defined in~\eqref{def:P}. A well-known result is that the sequence $\{x^{(k)}\}$ generated using~\eqref{eq:PnP} converges if $P(t)$ is stable~\cite{meyer2000matrix}. This explains the origin of \Cref{prob:existence} and its connection to the stability considerations in \Cref{sec:intro}. It was demonstrated in \cite{sreehari2016plug} that excellent results are obtained if the linear denoiser $W$ is modeled on the non-local means filter~\cite{buades2005review}. We briefly remark on the construction of $W$ and its mathematical properties, which motivate the specific assumptions on $W$ in~\Cref{sec:intro}. 

\begin{remark}[kernel denoiser]
\label{rmk:constructionW}
The linear denoiser $W$ is constructed from a kernel matrix $K \in \bbR^{n \times n}$, where $K_{ij}$ quantifies the proximity between pixels $i$ and $j$. The precise form of $K$ is not essential in the present context; we only require that $K$ is nonnegative, primitive, and positive semidefinite. The denoiser $W$ is obtained by row-normalizing $K$, producing a stochastic, primitive matrix. However, the normalization typically makes $W$ nonsymmetric. We refer to this construction as a kernel denoiser, which falls within the class of $W$ considered in~\Cref{sec:intro}.
\end{remark}

A procedure for constructing a symmetric kernel denoiser from $K$ was proposed in~\cite{sreehari2016plug}, resulting in a matrix $W$ that is symmetric, stochastic, and positive semidefinite. In this case, $W$ can be interpreted as the proximal operator of a convex potential~\cite{moreau1965proximite}, and the convergence of~\eqref{eq:PnP} then follows directly from standard optimization theory~\cite{Beck2017}. However, this argument does not yield a convergence rate. It was later shown in~\cite{ACK2023-contractivity} that linear convergence can be established since the operator $P(t)$ in~\eqref{eq:PnP} is contractive in this setting. 

To our knowledge, establishing convergence for a nonsymmetric $W$ is an open problem. This also has computational implications, as constructing the symmetric form of $W$ in~\cite{sreehari2016plug} is more expensive, even though it delivers comparable reconstruction quality. 
Importantly, 
such symmetrizations can result in worse reconstructions in some cases. 
These considerations motivated us to revisit the problem of establishing convergence without requiring $W$ to be symmetric.

\vspace{0.5em}

Building on the above discussion, we state the convergence results for the plug-and-play iterations in~\eqref{eq:PnP}.

\begin{corollary}
\label{corr:convergencePnP}
Suppose $W$ is a kernel denoiser and $A$ corresponds to inpainting, deblurring, or superresolution. Then there exists some $T>0$ such that the iterations generated by \eqref{eq:PnP} converge for any step size $ 0 < t < T$.
\end{corollary}

\begin{proof}
This follows from~\Cref{corr1:existence} and the properties of $B$ and $W$ in~\Cref{rmk:propertiesB,rmk:constructionW}.
\end{proof}

\begin{corollary}
\label{corr:convergenceInp}
Suppose $W$ is a kernel denoiser and $A$ corresponds to inpainting. Then the iterations generated by \eqref{eq:PnP} converge for any step size $ 0 < t < 2$.
\end{corollary}

\begin{proof}
This follows from~\Cref{rmk:propertiesB,rmk:constructionW},~\Cref{prop:inpainting}, and the fact that $\varrho(B)=1$ for inpainting~(see~\Cref{rmk:propertiesA,rmk:propertiesB}).
\end{proof}

\begin{remark}
It is easy to verify from~\Cref{rmk:propertiesA} that when $A$ corresponds to deblurring, the matrix $B=A^\top\!A$ in~\Cref{rmk:propertiesB} satisfies the assumptions $Be \leqslant \varrho(B) e$ and $\pi^\top Be > 0$ in~\Cref{conjecture}. Hence, a positive resolution of~\Cref{conjecture} would yield a convergence guarantee for deblurring for any step size $0 < t < 2/\varrho (B)$.
\end{remark}

The above ideas apply to the stability of $R(t)$ and the convergence of the associated iterative algorithm. However, exploring this further would take us away from the main topic, so we instead refer the interested reader to~\cite{ACK2023-contractivity}.

\section{Proofs}
\label{sec:proofs}

In this section, we prove \Cref{thm:existence,thm:doublystochastic} and \Cref{prop:inpainting,prop:nondiagB}. Some proofs require additional notation, concepts, and lemmas, which are introduced as needed. 

\subsection{Proof of~\Cref{thm:existence}}

\vspace{1em}

In this part, we work with complex-valued matrices and vectors and follow the notation in~\cite{greenbaum2020}. We write $v^*$ for the complex conjugate of $v \in \bbC^n$. If $\lambda \in \sigma(M)$ is simple (algebraic multiplicity $1$), then there exist $v \in \bbC^n$ (left eigenvector) and $u \in \bbC^n$ (right eigenvector) such that $v^* M = \lambda v^*, \, Mu= \lambda u,$ and $v^* u =1$.  We call $\lambda \in \sigma(M)$ \emph{dominant} if $\varrho(M)= |\lambda|$ and if $|\mu| < \varrho(M)$ for all $\mu \in \sigma(M), \, \mu \neq \lambda$. 

We denote by $D_r(0)$ the open disk $\{z \in \bbC: \ |z| < r\}$ around the origin. A map $M: D_r(0) \to \bbC^{n \times n}$ is said to be \emph{analytic} if each component $M_{ij}(z)$ is analytic (complex differentiable) on $D_r(0)$~\cite{kato1995}.

To apply perturbation theory of eigenvalues, we extend the domain of \eqref{def:P} and \eqref{def:R} by replacing the real parameter $t$ in \eqref{def:P} and \eqref{def:R} with a complex parameter $z$. Specifically, we define the maps $P: \bbC \to \bbC^{n \times n}$ and $R: \bbC \to \bbC^{n \times n}$ by (reusing the same notation for simplicity)
\begin{equation}
\label{eq:PR}
P(z) = W(I-zB),  \qquad R(z) = I-W + (I+zB)^{-1} (2W-I),
\end{equation}
where $R$ is defined in a neighborhood of the origin. We can view \eqref{def:P} and \eqref{def:R} as restrictions of \eqref{eq:PR} to the real line. 

\begin{proposition}
\label{prop:analyticPR}
The functions in~\eqref{eq:PR} are analytic around the origin, and 
\begin{equation}
\label{eq:derivative}
P'(0) = - WB, \qquad  R'(0) = - B  (2W-I).
\end{equation}
\end{proposition}

\begin{proof}
The function $P(z) = W-zWB$ is analytic on the entire complex plane, with derivative $P'(z)=-WB$. On the other hand, we have
\begin{equation*}
R(z) =  I-W + G(z)^{-1} (2W-I), \qquad G(z):=I+zB.
\end{equation*}
The function $G(z)$ is analytic with $G'(z)=B$. Since $G(0)=I$ is invertible, there exists $r>0$ such that $G(z)$ is invertible on $D_r(0)$. Moreover, its inverse $G(z)^{-1}$ is analytic on $D_r(0)$ with derivative~(e.g., see~\cite[Chapter~9]{lax2007linear})
\begin{equation}
\label{eq:deriv}
\frac{d}{dz}  G(z)^{-1} = -G(z)^{-1} G'(z)\,G(z)^{-1}  = -(I+zB)^{-1} B \, (I+zB)^{-1} .
\end{equation}
It follows that $R(z)$ is differentiable on $D_r(0)$, with derivative 
\begin{equation*}
R'(z) =\frac{d}{dz}  G(z)^{-1} (2W-I).
\end{equation*}
Combining this with~\eqref{eq:deriv}, we get the formula for $R'(0)$.
\end{proof}

\begin{lemma}
\label{lemma:perturbationrho}
Let $M(z)$ be a $\bbC^{n \times n}$-valued map that is analytic in a neighborhood of the origin. Suppose $M(0)$ has a simple,  dominant eigenvalue $\lambda_0$ with left and right eigenvectors $v$ and $u$ such that $v^* u =1$. Then there exists a neighborhood of the origin on which the function $\varrho(z)=\varrho(M(z))$ is analytic, and 
\begin{equation}
\label{eq:derivrho}
\varrho'(0) = v^* M'(0) \, u.
\end{equation}
\end{lemma}

\begin{proof}
Since $\lambda_0$ is a simple eigenvalue of $M(0)$, there exists $r > 0$ and an analytic function $\lambda: D_r(0) \to \bbC$ such that (e.g., see~\cite[Theorem~1]{greenbaum2020})
\begin{equation}
\label{eq:derivlambda}
\lambda(0) = \lambda_0,
\quad
\lambda'(0) = v^* M'(0) u,
\end{equation}
where $u$ and $v$ are as defined in~\Cref{lemma:perturbationrho}. Moreover, we know that $\varrho(0) = \lambda_0$, where $\lambda_0$ is an isolated eigenvalue of $M(0)$. Since the eigenvalues of $M(z)$ depend continuously on the components of $M(z)$ and since $M(z)$ is continuous, the eigenvalues depend continuously on $z$. Consequently, by shrinking $r$ if necessary, we have $\varrho(z) = \lambda(z)$ for all $z \in D_r(0)$.
\end{proof}

\begin{proof}[Proof of \Cref{thm:existence}]
Note that $P(z)$ and $R(z)$ in~\eqref{eq:PR} satisfy the assumptions of~\Cref{lemma:perturbationrho}. Moreover, $P(0)=R(0)=W$, where $W$ is assumed to be primitive and stochastic. Consequently, $\lambda_0=1$ is a simple, dominant eigenvalue of both $P(0)$ and $R(0)$, with corresponding left and right eigenvectors $v=\pi$ and $u=e$ such that $\pi^\top \! e=1$.  It then follows from~\Cref{lemma:perturbationrho} that the functions $ \varrho(P(z))$ and $ \varrho(R(z))$ are analytic in a neighborhood of $0$. Also as $\pi^\top B e>0$, we obtain from \Cref{prop:analyticPR} that
\begin{equation*}
\frac{d  \varrho(P(z)) }{dz} \,\bigg|_{z=0}= \pi^\top P'(0) \, e = - \pi^\top W B e = - \pi^\top \! B e < 0,
\end{equation*}
and
\begin{equation*}
\frac{d  \varrho(R(z)) }{dz} \,\bigg|_{z=0}= \pi^\top R'(0) \, e = - \pi^\top B \, (2W-I)\, e = -\pi^\top\! B e < 0.
\end{equation*}
In particular, this implies that there exists $T>0$ such that $\varrho(P(t))$ and $\varrho(R(t))$ are strictly decreasing on the interval $(-T, T)$. Since $ \varrho(P(0)) = \varrho(R(0)) = 1$, the claim follows.
\end{proof}

\subsection{Proof of~\Cref{thm:doublystochastic}}

\vspace{1em}

In this section, $\|v\|_2$ denotes the standard Euclidean norm of $v \in \bbC^n$, and $\|M\|_2$ denotes  the spectral norm of $M \in \bbC^{n \times n}$. We know that $\varrho(M) \leqslant \|M\|_2$ for any $M \in \bbC^{n \times n}$, with equality if $M$ is Hermitian~\cite{meyer2000matrix}.

\begin{lemma}
\label{lemma:Cmatrix}
Suppose $B \in \bbR^{n \times n}$ be positive semidefinite, $\lambda \in \bbC$ with $|\lambda| \geqslant 1$, and $t > 0$. Then $C :=\lambda I + t(\lambda -1) B$ is invertible and $\|C^{-1}\|_2 \leqslant 1$.
\end{lemma}

\begin{proof}
The eigenvalues of $C$ are of the form $\lambda + t(\lambda - 1)\mu$ for some $\mu \in \sigma(B)$. Since $B$ is positive semidefinite, $\mu \geqslant 0$. If $C$ were singular, we would have $\lambda + t(\lambda - 1)\mu = 0$ for some $t > 0$. However, this implies $\lambda = t\mu/(1+t\mu)$, so that $|\lambda| = \lambda <  1$, contradicting our assumption. Hence, $C$ must be invertible.

Since $C^{-1}$ is Hermitian, we have $\|C^{-1}\|_2 =\varrho(C^{-1})$. The eigenvalues of $C^{-1}$ are of the form $1/(\lambda (1+ \mu t) - \mu t), \, \mu \in \sigma(B)$, and
\begin{equation}
\label{eq:comeslater}
|\lambda (1+ \mu t) - \mu t | \geqslant |\lambda|  (1+\mu t) - \mu t \geqslant 1+\mu t - \mu t = 1.
\end{equation}
Hence, $\varrho(C^{-1}) \leqslant 1$.
\end{proof}

\begin{lemma}
\label{lemma:doublystochastic}
Let $W$ be doubly stochastic. Then $\|Wv\|_2 \leqslant \|v\|_2$ for all $v \in \bbC^n$.
Moreover, if $W^\top W$ is irreducible, then equality holds only if $v$ is a scalar multiple of $e$.
\end{lemma}

\begin{proof}
Since both $W^\top$ and $W$ are stochastic, we have $\varrho(W^\top W) \leqslant \varrho(W^\top) \varrho(W)=1$. Hence, for any $v \in \bbC^n$, 
\begin{equation*}
\|Wv\|_2^2 = v^* (W^\top W v) \leqslant \varrho(W^\top W) \|v\|_2^2 \leqslant \|v\|_2^2.
\end{equation*}

Now, suppose $\|Wv\|_2 = \|v\|_2$ for some $v \in \bbC^n$. We can write that is 
\begin{equation}
\label{eq:qf}
v^* S v = v^* v, \qquad S:=W^\top W.
\end{equation}
Since $S$ is positive semidefinite and $\varrho(S) \leqslant 1$, we have $\sigma(S) \subset [0,1]$. Combining this with \eqref{eq:qf} and using the spectral theorem for Hermitian matrices, we can deduce that $Sv=v$.
Finally, because $S$  is stochastic and irreducible, it follows that $v= c e$ for some $c \in \bbC$.
\end{proof}

\begin{proof}[Proof of \Cref{thm:doublystochastic}]
\label{proof:doublystochastic}
Since $I-tB$ is symmetric and $B$ is positive semidefinite, we have for all  $0< t < 2/\varrho(B)$,
\begin{equation}
\label{eq:normB}
\|I-tB\|_2= \varrho(I-tB) \leqslant 1.
\end{equation}
Combining this with \Cref{lemma:doublystochastic}, it follows that for all $0< t < 2/\varrho(B)$,
\begin{equation*}
\varrho(P(t))  \leqslant  \|P(t) \|_2 \leqslant \|W \|_2 \|(I-tB) \|_2 \leqslant 1.
\end{equation*}

To complete the proof, we must show that if $0< t < 2/\varrho(B)$, then $| \lambda | < 1$ for all $\lambda \in \sigma(P(t))$. 
Suppose, to the contrary, that there exists $\lambda \in \sigma(P(t))$ with $|\lambda| = 1$. Then, for some left eigenvector $v \in \bbC^n$,
\begin{equation}
\label{eq:eigeqn1}
\lambda^* v = P(t)^\top v = (I-tB) W^\top v.
\end{equation}
Taking norms on both sides and using \eqref{eq:normB}, we get $\|v\| \leqslant \|W^\top v\|$. Since $W^\top$ is doubly stochastic and $WW^\top$ is assumed to be irreducible, it follows from~\Cref{lemma:doublystochastic} that $v=c e$ for some $c \in \bbC$. Substituting this into \eqref{eq:eigeqn1} yields
\begin{equation}
\label{eq:eigeqn2}
\lambda^* e =  (I-tB) W^\top e = e-tBe.
\end{equation}
Thus, $(1-\lambda)/t \in \sigma(B)$. Since $B$ is symmetric, the only possibilities are $\lambda = \pm 1$. However, if $\lambda=1$, we have $Be=0$ from \eqref{eq:eigeqn2}, which contradicts our assumption. On the other hand, $\lambda=1$ yields $Be=(2/t) e$, so that $2/t \leqslant \varrho(B)$. This contradicts the condition $0<t < 2 /\varrho(B)$. Thus, we must have $| \lambda | < 1$ for all $\lambda \in \sigma(P(t))$.

We will prove using contradiction that $\varrho(R(t)) < 1$ whenever $0< t < 2/\varrho(B)$. Suppose, to the contrary, that there exists $\lambda \in \sigma(R(t))$ with $|\lambda| \geqslant 1$.  Then there exists a right eigenvector $v \in \bbC^n$ such that
\begin{equation}
\label{eq:eigeqn3}
\lambda v =  (I-W)v + (I+tB)^{-1}(2W-I)v,
\end{equation}
where the inverse exists for $t >0$. Rearranging gives $(I-tB)Wv = Cv$, where $C$ is as defined in~\Cref{lemma:Cmatrix}. Since $t > 0$, \Cref{lemma:Cmatrix} guarantees that $C^{-1}$ exists, so that
\begin{equation}
\label{eq:usedlater}
v = C^{-1} (I - tB)Wv.
\end{equation}
Taking norms on both sides and using \eqref{eq:normB} together with the bound $\|C^{-1}\|_2 \leqslant 1$ from~\Cref{lemma:Cmatrix}, we obtain
$\|v\|_2 \leqslant \|Wv\|_2$. Since $W$ is doubly stochastic, it follows from~\Cref{lemma:doublystochastic} that $\|v\|_2 = \|Wv\|_2$. Moreover, as $W^\top W$ is irreducible, ~\Cref{lemma:doublystochastic} implies that $v= c e$ for some $c \in \bbC$. Substituting this into \eqref{eq:eigeqn3} yields
\begin{equation}
\label{eq:eigeqn4}
e = \lambda  e + t \lambda Be.
\end{equation}
The rest of the argument is similar to that for $P(t)$. It follows from \eqref{eq:eigeqn4} that $\lambda \in \bbR$, so that $\lambda = \pm 1$. However, $\lambda = 1$ yields $Be=0$, and $\lambda = -1$ yields $\varrho(B) \geqslant 2/t$, both of which  contradict our assumptions. Therefore, we must have $| \lambda | < 1$ for all $\lambda \in \sigma(R(t))$.
\end{proof}

\subsection{Proof of~\Cref{prop:inpainting}}

\vspace{1em}

We need the following notations. For $v \in \bbC^n$ and $M \in \bbC^{n \times n}$, we define $|v| \in \bbR^n$ and $|M| \in \bbR^{n \times n}$ by taking absolute values componentwise. For $M, N \in \bbR^{n \times n}$, we write $M \leqslant N$ to mean $M_{ij} \leqslant N_{ij}$ for all $i,j$. Finally, for $v \in \bbC^n$, we write $\diag(v) \in \bbC^{n \times n}$ for the diagonal matrix whose diagonal entries are the components of $v$.

\begin{lemma}
\label{lemma:superstochastic}
Let $W \geqslant 0$ be irreducible with $\varrho(W) = 1$. Suppose $p \geqslant 0$ satisfies $p \leqslant Wp$. Then $p = Wp$
and $p >0$.
\end{lemma}

\begin{proof}
See~\cite[Chapter 2]{berman1994}.
\end{proof}

\begin{lemma}
\label{lemma:inpainting}
Let $W \geqslant 0$ be irreducible with \(\varrho(W)=1\). Let $D \in \bbC^{n\times n}$ be a diagonal matrix with $|D| \leqslant I$. Suppose there exists $v \neq 0$ such that either $WDv = v$ or $DWv=v$. Then necessarily $|D|=I$.
\end{lemma}

\begin{proof}
Suppose there exists $v \neq 0$ such that $WDv=v$. Let $p=|v|$. Since $|D| \leqslant I$, we have
\begin{equation}
\label{eq:temp1}
p=|v| =  |WDv| \leqslant W |D| p \leqslant W p.
\end{equation}
Because $p \geqslant 0$, it follows from~\Cref{lemma:superstochastic} that $p=Wp$ and $p>0$. Substituting this into \eqref{eq:temp1} yields
\begin{equation}
\label{eq:lemma:inpainting:1}
p = W |D| p.
\end{equation}

Multiplying \eqref{eq:lemma:inpainting:1} by the left Perron vector of $W$ gives $\pi^\top p = \pi^\top W |D| p = \pi^\top |D| p$, so that $\pi^\top (I - |D|)p = 0$. As $\pi, p > 0$ and $|D|\leqslant I$, this forces $|D| = I$. An identical argument applies if $DWv=v$.
\end{proof}

\begin{proof}[Proof of \Cref{prop:inpainting}]
Suppose there exists $\lambda \in \sigma(P(t))$ with $|\lambda| \geqslant 1$, so that  $\lambda v = P(t)v$ for some nonzero $v \in \bbC^n$. Since $\lambda \neq 0$, we can write this as  
\begin{equation*}
v=WDv, \qquad D:= \frac{1}{\lambda} (I - tB).
\end{equation*}

Let $B=\diag(b)$, where $b \geqslant 0$ and $b \neq 0$ by assumption. Because $|\lambda| \geqslant 1$, we have $|D| \leqslant I$ whenever $0 < t < 2/\varrho(B)$. By \Cref{lemma:inpainting}, this forces $|D|=I$. Since $b \neq 0$, there exists $i$ such that $b_i >0$. For this $i$, we cannot have $|1-tb_i| = |\lambda|$, because $|1-tb_i| < 1$ for $0 < t < 2/\varrho(B)$. This shows we cannot have $|\lambda| \geqslant 1$ for any $\lambda \in \sigma(P(t))$.

Similarly, suppose $\lambda \in \sigma(R(t))$ with $|\lambda| \geqslant 1$, and $v \neq 0$ so that \eqref{eq:usedlater} holds. Since $B=\diag(b)$, we can write 
\begin{equation}
\label{eq:temp2}
v= DWv, \qquad D:=\diag(d), \: d_i = \frac{1-tb_i}{\lambda+t(\lambda-1)b_i}.
\end{equation}
From \eqref{eq:comeslater}, it follows that for $t>0$,
\begin{equation}
\label{eq:gt1}
|\lambda+t(\lambda-1)b_i| \geqslant 1.
\end{equation}
Moreover if $0 < t < 2/\varrho(B)$, then $|1-tb_i| \leqslant 1$. Hence, from \eqref{eq:temp2}, we obtain $|D| \leqslant I$. By~\Cref{lemma:inpainting}, this forces $|D| = I$.  However, we know that $b_i >0$ for some $i$. For such as $i$, 
\begin{equation}
|\lambda+t (\lambda-1)b_i| = |1-tb_i| < 1,
\end{equation}
which contradicts \eqref{eq:gt1}. Thus, we must have $|\lambda| < 1$ for any $\lambda \in \sigma(R(t))$.
\end{proof}

\subsection{Proof of~\Cref{prop:nondiagB}}

\vspace{1em}

We begin by noting that if $B=\alpha I + \beta E$, then $\alpha, \alpha + n\beta \in \sigma(B)$. Since $B$ is positive semidefinite and $Be \neq 0$, it follows that $\alpha \geqslant 0$ and $\alpha + n\beta > 0$.  Hence,
\begin{equation}
\label{eq:boundrho}
\varrho(B) \geqslant \max  \big( \alpha, \alpha + n\beta \big).
\end{equation}

Let $\lambda \in \sigma(P(t))$. Then there exists a left eigenvector $u \in \bbC^n$ such that 
\begin{equation}
\label{eq:eigen}
\lambda u^* = u^* W(I - tB).
\end{equation}
Multiplying on the right by $e$ gives $\lambda u^*e=(1-\alpha t-n\beta t ) \, u^* e$. Thus, either $\lambda = 1-t(\alpha + n\beta)$ or $u^*e=0$. In the first case, it follows from~\eqref{eq:boundrho} that $|\lambda | < 1$ for all $0 < t < 2/\varrho(B)$. In the second case, since $EW^\top u=0$, taking the complex conjugate of \eqref{eq:eigen} yields
\begin{equation}
\label{eq:temp3}
\lambda^* u = (I - tB)W^\top u = (1-\alpha t)W^\top  u.
\end{equation}
If $1 - \alpha t = 0$, then trivially $|\lambda| < 1$. Otherwise, $\lambda^*/(1 - \alpha t) \in \sigma(W^\top)$. We consider the following cases:
\begin{itemize}
\item If $\alpha > 0$, then as $\varrho(W^\top) \leqslant 1$, it follows from \eqref{eq:boundrho} that $|\lambda| \leqslant |1-\alpha t| < 1$ for all $0 < t < 2/\varrho(B)$.
\item If $\alpha = 0$, then \eqref{eq:temp3} reduces to $\lambda u^* = u^*W$. Suppose $|\lambda| = 1$. Because $W$ is stochastic and primitive, this forces $u= c \pi, c \in \bbC$, where $\pi$ is the left Perron vector of $W$. However, since  $u^* e = 0$ and $\pi^\top e > 0$, we get $c = 0$, contradicting the assumption that $u \neq 0$.
\end{itemize}
Thus, in all cases, $|\lambda| < 1$. Since this holds for any $\lambda \in \sigma(P(t))$, we conclude that $\varrho(P(t)) < 1$.

\vspace{1em}

Now, let $\lambda\in\sigma(R(t))$. Then there exists a nonzero  $u\in\mathbb C^n$ such that
\begin{equation}
\label{eq:eigenR}
\lambda\,u^*  =  u^* R(t) =  u^*(I-W) + u^*(I+tB)^{-1}(2W-I).
\end{equation}
Right-multiplying \eqref{eq:eigenR} by $e$ gives
\begin{equation*}
\lambda\,u^*e  =  (1+\alpha t+n\beta t)^{-1}\,u^*e.
\end{equation*}
Hence, either $\lambda=1/(1+(\alpha+n\beta)t)$ or $u^*e=0$. In the first case, since $\alpha+n\beta > 0$, we have $|\lambda|= \lambda< 1$ for all $t >0$. In the second case, taking the complex conjugate of \eqref{eq:eigenR} yields
\begin{equation*}
\lambda^* u  =  (I-W^\top)u + (2W^\top-I)(I+tB)^{-1}u. 
\end{equation*}
Since $u^*e=0$, we have $(I+tB)^{-1}u=u/(1+\alpha t)$. Substituting this, we obtain
\begin{equation*}
\lambda^* u  =  \frac{1}{1+\alpha t}\big(\alpha t\,I + (1-\alpha t)W^\top\big)u.
\end{equation*}
From this, we conclude that 
\begin{equation*}
\lambda  =  \frac{\alpha t + (1-\alpha t)\mu^*}{1+\alpha t}
\end{equation*}
for some $\mu \in \sigma(W^\top)$. As before, since $W$ is primitive and stochastic and $u^*e=0$, we must have $|\mu|<1$. If $\alpha = 0$, then $\lambda=\mu^*$, so that $|\lambda|<1$. On the other hand, if $\alpha > 0$, then for any $0 < t < 2/\varrho(B)$ we have from \eqref{eq:boundrho} that $0 < \alpha t < 2$. Thus, 
\begin{equation*}
   |\lambda|
     =  \frac{|\,\alpha t + (1-\alpha t)\mu\,|}{1+\alpha t}
    \  \leqslant\ \frac{\alpha t + |1-\alpha t|\,|\mu|}{1+\alpha t}
    \ <\ \frac{\alpha t + |1-\alpha t|}{1+\alpha t} < 1.
\end{equation*}
Therefore, in all cases, $|\lambda|<1$. Since this holds for any $\lambda\in\sigma(R(t))$, we conclude that $\varrho(R(t))<1$.

\section*{Acknowledgment}

The authors thank Arghya Singha for helpful discussions and clarifications. K.~N.~Chaudhury acknowledges partial support from grant STR/2021/000011, ANRF, Government of India.

\bibliographystyle{cas-model2-names}
\bibliography{references}

\end{document}